\documentclass[a4paper,reqno,11pt]{amsart}
\usepackage{amsfonts,amsmath,amsthm,amssymb,stmaryrd}
\usepackage{hyperref}
\usepackage{color}
\usepackage{mathrsfs}

\newtheorem{Theorem}{Theorem}[section]
\newtheorem{Lemma}[Theorem]{Lemma}
\newtheorem{Proposition}{Proposition}[section]

\newtheorem{Definition}{Definition}[section]
\numberwithin{equation}{section} \allowdisplaybreaks
\allowdisplaybreaks 
\begin{document}
\author{Yi Peng}
\address{College of Mathematics and Statistics, Chongqing University,
                             Chongqing, 401331,  China.}
\email[Y. Peng]{20170602018t@cqu.edu.cn}
\author{Huaqiao Wang}
\address{College of Mathematics and Statistics, Chongqing University,
                             Chongqing, 401331,  China.}
\email[H.Q. Wang]{wanghuaqiao@cqu.edu.cn}

\title[Strong solutions of the LLB equation in Besov space]
{Strong solutions of the Landau--Lifshitz--Bloch equation in Besov space}
\thanks{Corresponding author: wanghuaqiao@cqu.edu.cn}
\

\begin{abstract}
We focus on the existence and uniqueness of the three--dimensional Landau--Lifshitz--Bloch equation supplemented with the initial data in Besov space $\dot{B}_{2,1}^{\frac{3}{2}}$. Utilizing a new commutator estimate, we establish the local existence and uniqueness of strong solutions for any initial data in $\dot{B}_{2,1}^{\frac{3}{2}}$. When the initial data is small enough in $\dot{B}_{2,1}^{\frac{3}{2}}$, we obtain the global existence and uniqueness. Furthermore,  we also establish a blow--up criterion of the solution to the Landau--Lifshitz--Bloch equation and then we prove the global existence of strong solutions in Sobolev space under a new condition based on the blow--up criterion.
\end{abstract}

\maketitle
\section{Introduction}
In recent decades, there have been many research results for micro--magnetic models due to their applications in industry and physics. One of those basic models is Landau--Lifshitz--Gilbert (LLG for short) equation (see \cite{G,LL}), which governs the motion of the magnetic spin below the Curie temperature. A simple version of the LLG equation is given by
\begin{equation}\label{LLG}
\begin{cases}
\partial_tm=-m\times\Delta m-\lambda m\times m\times \Delta m,\\
m|_{t=0}=m_0
\end{cases}
\end{equation}
where $m: \mathbb{R}^3\times\mathbb{R}^+\rightarrow \mathbb{S}^2$ denotes the spin director field, $\lambda>0$ is a Gilbert damping parameter.

There have been many research results about the well--posedness for the LLG equation. For instance, Huber \cite{H} studied the periodic solution to the LLG equation by perturbation argument and the spectral analysis on the corresponding linearized system. Under the smallness condition of initial gradients in the $L^n$ norm, Melcher \cite{M} obtained the existence, uniqueness and asymptotics of global smooth solutions for the LLG equation in dimension $n\geqslant3$. Later, Lin--Lai--Wang \cite{LLW} established the global well--posedness of the LLG equation in $\mathbb{R}^n$ ($n\geqslant2$) provided the initial gradients is small in a Morrey space. Recently, Feischl--Tran \cite{FT} proved that while the initial data is sufficiently close to a constant function, the three--dimensional LLG equation with homogeneous  Neumann  boundary  condition admits arbitrarily  regular solutions. Guti\'{e}rrez--De Laire \cite{GD} built the global existence of the self--similar solutions to LLG equation in any dimension under the hypothesis that the BMO semi--norm of the initial data is small. Di Fratta--Innerberger--Praetorius \cite{DIP} proved the weak--strong uniqueness of the LLG equation on a bounded domain in $\mathbb{R}^3$. Very recently, Wang--Guo \cite{WG} deduced a blow--up criterion for the LLG equation on bounded domain with Neumann boundary condition in n--dimensional ($n\geqslant2$) space. For the results of the stochastic LLG equation, we refer to \cite{BGJ,BGJ1,BL,BM,BMM} and the references therein.

While the electronic temperature is higher than the Curie temperature, this model must be replaced by Landau--Lifshitz--Bloch (LLB for short) equation (see \cite{GD1,GD,L}):
\begin{align}\label{LLB0}
\frac{\partial u}{\partial t}=\gamma u\times \mathcal{H}_{{\rm eff}}(u)+\frac{L_1}{|u|^2}(u\cdot\mathcal{H}_{{\rm eff}}(u))u-\frac{L_2}{|u|^2}u\times(u\times\mathcal{H}_{{\rm eff}}(u)),
\end{align}
where $|\cdot|$ is the Euclidean norm in $\mathbb{R}^3$, $u(t,x)\in \mathbb{R}^{3}$ is the magnetization vector, $\gamma>0$ denotes the gyromagnetic ratio, $L_1$ is the longitudinal damping parameter, $L_2$ is the transverse damping parameter, $\mathcal{H}_{{\rm eff}}(u)$ stands for the effective field witch is given by
$$\mathcal{H}_{{\rm eff}}(u)=\Delta u-\frac{1}{\chi_{||}}\left(1+\frac{3T}{5(T-T_c)|u|^2}\right)u,$$
where $\chi_{||}$ is the longitudinal susceptibility, $T_c$ denotes the Curie temperature. The LLB equation \eqref{LLB0} interpolates between the LLG equation and the Bloch equation, and it is able to govern the dynamical behavior of magnetic spin at all temperature. Setting $L_1=L_2=k_1$, the equation \eqref{LLB0} becomes
\begin{align}\label{LLB00}
\partial_tu=k_1\Delta u-k_2 u+\gamma u\times \Delta u-k_2\mu|u|^2u,
\end{align}
where $k_2=\frac{k_1}{\chi_{||}}$ and $\mu=\frac{3T}{5(T-T_c)}$.

Many authors were interested in the mathematical theories of the LLB equation. Le \cite{L} proved the existence of weak solution and also discussed the regularity properties. In terms of the following Generalized LLB equation
\begin{align*}
\partial_tu=k_1\Delta u+\gamma\nabla F(u)\times \Delta u-k_2(1+\mu\cdot F(u))\nabla F(u),
\end{align*}
for some $F\in \mathcal{C}^3(\mathbb{R}^3)$, Jia \cite{J} proved the local existence of a strong solution by using the Faedo--Galerkin approximation. Jia--Guo \cite{JG} discussed the existence and uniqueness of the LLB equation on m--dimensional closed Riemannian manifold. Later, Guo--Li \cite{GL} established the global existence of smooth solution in one--dimensional space. After that, Li--Guo--Zeng \cite{LGZ} obtained the global existence of smooth solution in $\mathbb{R}^2$ for any initial data, and in $\mathbb{R}^3$ for small initial data. Recently, Ayouch--Benmouane--Essouf \cite{ABE} established the uniqueness and local existence of the LLB equation in a bounded domain of $\mathbb{R}^3$. For the results of the stochastic LLB equation, see \cite{BGL,GM,JJW,QTW} and the references therein.

Since the LLG equation \eqref{LLG} is invariant under the following scaling transform:
$$m(x,t)\rightarrow m(\lambda x,\lambda^2t),\quad m_0(x)\rightarrow m_0(\lambda x),$$
for $\lambda>0$, $\dot{B}_{p,1}^{\frac{3}{p}}$ is a critical space for initial data. Guo--Huang\cite{GH} established the global well--posedness of the LLG equation with small initial data in critical space, and they also justified this global solution converges to that of Schr\"{o}dinger maps as the Gilbert damping term vanishes. Taking into account this fact, in the present paper, we consider the existence and uniqueness of solution in $\dot{B}_{2,1}^{\frac{3}{2}}$ to the following initial value problem:
\begin{equation}\label{LLB}
\begin{cases}
\partial_tu=\Delta u-\kappa u+u\times \Delta u-\kappa\mu|u|^2u,\\
u|_{t=0}=u_0,
\end{cases}
\end{equation}
where we suppose that the temperature is higher than Curie temperature, which means the coefficients in \eqref{LLB} are positive. Before stating our main results, let's present a few notations.

\textbf{Notation.} For any positive $A$ and $B$, we use the notation $A\lesssim B$ to mean that there exists a positive constant $C$ such that $A\leqslant CB$. And $A\thicksim B$ means $C_1A\leqslant B\leqslant C_2A$ for positive constants $C_1$, $C_2$. For every $p\in[1,\infty]$, let $\left\|\cdot\right\|_{L^{p}}$ denote the norm in the Lebesgue space $L^{p}$. For any normed space $X$, we employ the notation $L^{p}\left([0,T],X\right)$ to denote the space of functions $f$ such that for almost all $t\in (0,T)$, $f(t)\in X$ and $\left\|f(t)\right\|_{X}\in L^{p}(0,T)$. We simply denote the notation $L^{p}_{T}X:=L^{p}\left([0,T],X\right)$. The set of bounded continuous functions from interval $I\subset \mathbb{R}$ to $X$ is denoted by $\mathcal{C}_b\left(I; X\right)$.

For the sake of convenience, we define the space
$$E_2(T):=\left\{f\in \mathcal{C}\left([0,T];B_{2,1}^{\frac{3}{2}}\right)\cap L^1\left([0,T],B_{2,1}^{\frac{7}{2}}\right)\cap L^1\left([0,T],B_{2,1}^{\frac{3}{2}}\right)\right\},$$
and its global version $E_2$ (with $f\in \mathcal{C}_b\left(\mathbb{R}^+; B_{2,1}^{\frac{3}{2}}\right)$) if $T=+\infty$.

Our first result state the local well--posedness to the LLB equation supplemented with large initial data in Besove space $\dot{B}_{2,1}^{\frac{3}{2}}$.
\begin{Theorem}\label{Thm1.1}
For any initial data $u_0\in \dot{B}_{2,1}^{\frac{3}{2}}$, there exists a positive time $T$ such that the Cauchy problem \eqref{LLB} has a unique strong solution $u\in E_2(T)$.
\end{Theorem}
Then the frame work of Theorem \ref{Thm1.1} enables us to obtain the following global result of the LLB equation with small initial data in $\dot{B}_{2,1}^{\frac{3}{2}}$.
\begin{Theorem}\label{Thm1.2}
If there exists a constant $\varepsilon>0$ such that $\|u_0\|_{\dot{B}_{2,1}^{\frac{3}{2}}}<\varepsilon$, then the Cauchy problem \eqref{LLB} has a unique global strong solution $u\in E_2$ satisfying
\begin{align}\label{Thm1.1-1}
\|u\|_{L^\infty\dot{B}_{2,1}^{\frac{3}{2}}}+\frac{C_1}{2}\|u\|_{L^1\dot{B}_{2,1}^{\frac{7}{2}}}+\frac{\kappa}{2}\|u\|_{L^1\dot{B}_{2,1}^{\frac{3}{2}}}\leqslant\varepsilon,
\end{align}
for some positive constant $C_1$.
\end{Theorem}
Our third goal of this paper is to obtain the global well--posedness of the LLB equation in Sobolev space under a new condition by establishing a blow--up criterion.
\begin{Theorem}[Blow-up]\label{Thm1.3}
Suppose $u_0\in H^m$,  $m\geqslant 2$. For the first blow--up time $T^*<\infty$ of the strong solution to the equation \eqref{LLB}, we have
\begin{align*}
\int_0^{T^*}\|u\|^{2}_{\dot{B}^{\frac{3}{p}}_{p,1}}+\|u\|^{\frac{2}{2-\delta}}_{\dot{B}^{2-\delta}_{\infty,\infty}}dt=\infty,
\end{align*}
where $\delta\in (1,2)$.
\end{Theorem}

\begin{Theorem}\label{Thm1.4}
Under the hypothesis of Theorem \ref{Thm1.3}, if there exists a small $\eta>0$ such that $\|u_0\|_{\dot{B}_{p,1}^{\frac{3}{p}}}<\eta$ ($1<p<\infty$), then LLB equation \eqref{LLB} admits a unique global solution $u\in \mathcal{C}(0,\infty;H^m(\mathbb{R}^3))$.
\end{Theorem}

Now, we sketch the strategy of proving these four theorems and point out some of the main difficult and techniques involved in the process. To prove Theorem \ref{Thm1.1}, we divide the LLB equation \eqref{LLB} into two parts: the known one $u^L:=e^{t\Delta}u_0$, and $\tilde{u}=u-u^L$, and we focus on the solvability to $\tilde{u}$. We apply a priori estimate on the approximate solutions constructed by Friedrichs method to obtain the uniform bounds in $E_2(T)$, then the local existence of the strong solution in $E_2(T)$ is obtained by a standard compactness argument. In this process, the main difficult occurs in the a priori estimate. The basic estimate (that is the priori estimate of Theorem \ref{Thm1.2}), obtained by product estimate \eqref{product est} directly, fails to handle the nonlinear term $u\times\Delta u$. More specifically, the term $\|u^L\|_{L^\infty_T\dot{B}_{2,1}^{\frac{3}{2}}}\|\tilde{u}\|_{L^1_T\dot{B}_{2,1}^{\frac{7}{2}}}$ on the right--hand side will not be cancelled out by the dissipation term $\|\tilde{u}\|_{L^1_T\dot{B}_{2,1}^{\frac{7}{2}}}$ on the left--hand side through setting $T$ small enough. We develop a new commutator estimate established in Lemma \ref{Lemma4} on the nonlinear term $u\times\Delta u$ to overcome this difficult. We prove Theorem \ref{Thm1.2} by running a same process of Theorem \ref{Thm1.1}. Under the frame of energy method, combining the Bernstein--like inequality \eqref{new Bernstein ineq} and the blow--up criterion established in Theorem \ref{Thm1.3}, we prove the global existence of strong solution to the LLB equation in Sobolev space under a new smallness condition in Theorem \ref{Thm1.4}.

This paper is arranged as follows. In Section \ref{sec2}, we recall same basic definitions and propositions of Besov spaces, and we also list some useful lemmas. In Section \ref{sec3}, we prove the main results. More precisely, we provide the proof of Theorem \ref{Thm1.1} in Section \ref{sec3.1}, while the following subsection gives the proof of Theorem \ref{Thm1.2}. Finally, we prove Theorem \ref{Thm1.3} and Theorem \ref{Thm1.4} in Section \ref{sec3.4} and Section \ref{sec3.5}, respectively.

\section{Preliminaries}\label{sec2}
In this section, we recall the homogeneous Littlewood--Paley decomposition theory and some basic facts in Besov spaces, and we also list some useful lemmas.

The Fourier transform of $f\in L^1(\mathbb{R}^3)$ is defined by
$$\mathcal{F}f(\xi)=\hat{f}(\xi):=\int_{\mathbb{R}^{3}}e^{-2\pi ix\cdot\xi}f(x)dx,$$
and its inverse is defined by
$$\mathcal{F}^{-1}g(x)=\check{g}(x):=\int_{\mathbb{R}^{3}}e^{2\pi ix\cdot\xi}g(\xi)d\xi.$$
Let $\mathcal{S}(\mathbb{R}^3)$ be the Schwartz space, and then the Fourier transform of $f\in \mathcal{S}'$ is the tempered distribution $\hat{f}$ given by
$$<\hat{f},\varphi>=<f,\hat{\varphi}>,\;\;\varphi\in \mathcal{S}.$$
Especially, for any $f\in L^2(\mathbb{R}^3)$, we have $\hat{f}\in L^2(\mathbb{R}^3)$ with
$$\hat{f}(\xi)=\lim\limits_{M\rightarrow \infty}\int_{|x|<M}e^{-2\pi ix\cdot\xi}f(x)dx.$$
For more details, see \cite[Chapter 1]{DZ}.

Let $\mathcal{C}=\{\xi\in\mathbb{R}^{3}:\frac{3}{4}\leqslant|\xi|\leqslant\frac{8}{3}\}$, then there exist smooth radial functions $0\leqslant\varphi,\chi\leqslant1$, supported in $\mathcal{C}$ and $B(0,\frac{4}{3})$, respectively, such that
$$\chi(\xi)+\sum\limits_{q \geqslant0}\varphi(2^{-q}\xi)=1,\;\forall\xi\in\mathbb{R}^{3},$$
$$\sum\limits_{q \in\mathbb{Z}}\varphi(2^{-q}\xi)=1,\;\forall\xi\in\mathbb{R}^{3}\setminus\{0\}.$$
Setting $\varphi_{j}=\varphi(2^{-j}\xi)$ and $\chi_{j}=\chi(2^{-j}\xi)$, for any $u\in\mathcal{S'}(\mathbb{R}^{3})$, the homogeneoous dyadic blocks $\dot{\Delta}_{j}$ and the homogeneous low--frequency cut--off operators $\dot{S}_{j}$ are defined by
$$\dot{\Delta}_{j}f=\varphi(2^{-j}D)f=\left(\mathcal{F}^{-1}\varphi_{j}\right)\ast f,\;j\in\mathbb{Z},$$
$$\dot{S}_{j}=\chi(2^{-j}D)f=\left(\mathcal{F}^{-1}\chi_{j}\right)\ast f,\;j\in\mathbb{Z}.$$
For any $f\in \mathcal{S'}(\mathbb{R}^{3})\setminus \mathscr {P}=\mathcal{S'}_{h}(\mathbb{R}^{3})$, there holds
$$f=\sum\limits_{j\in\mathbb{Z}}\dot{\Delta}_{j}f,\;
\dot{S}_{j}f=\sum\limits_{j'\leqslant j-1}\dot{\Delta}_{j'}f,\;\forall j\in\mathbb{Z},$$
where $\mathscr{P}$ is the set of polynomials.

Now, let's recall the definition of the homogenous Besov spaces.
\begin{Definition}
Let $s\in\mathbb{R}$, $1\leqslant p,r\leqslant\infty$, the homogenous Besov spaces $\dot{B}_{p,r}^{s}$ are defined by
$$\dot{B}_{p,r}^{s}=\left\{f\in\mathcal{S'}_{h}(\mathbb{R}^{3}): \parallel f\parallel_{\dot{B}_{p,r}^{s}}<\infty\right\},$$
 where
\begin{equation*}
\parallel f\parallel_{\dot{B}_{p,r}^{s}}=\left\{\begin{array}{clcc}
\left(\sum\limits_{j\in\mathbb{Z}}2^{rjs}\left\|\dot{\Delta}_{j}f\right\|_{L^{p}}^{r}\right)^{\frac{1}{r}},& for&1\leqslant r<\infty,\\
\sup\limits_{j\in\mathbb{Z}}2^{js}\left\|\dot{\Delta}_{j}f\right\|_{L^{p}},& for&r=\infty.
\end{array}\right.
\end{equation*}
\end{Definition}
We also introduce the homogeneous Sobolev space $\dot{H}^{s}(\mathbb{R}^{3})$, for any $s\in \mathbb{R}$, as the subspace of tempered distributions whose Fourier transform is locally integrable and the following norm is finite:
$$\left\|f\right\|_{\dot{H}^{s}}=\left(\int_{\mathbb{R}^{3}}\left|\xi\right|^{2s}\left|\hat{f}(\xi)\right|^{2}d\xi\right)^{\frac{1}{2}}.$$
And the Sobolev space $H^s(\mathbb{R}^{3})$ ($s\in \mathbb{R}$) consists of tempered distributions $f$ such that $\hat{f}\in L^2_{loc}(\mathbb{R}^{3})$ and
$$\left\|f\right\|_{H^{s}}=\left(\int_{\mathbb{R}^{3}}(1+\left|\xi\right|^2)^{s}\left|\hat{f}(\xi)\right|^{2}d\xi\right)^{\frac{1}{2}}<\infty.$$

Next, we will list some basic facts on Besov spaces, which have been proved in \cite{BCD}.
\begin{Proposition}\label{p1}
The following propositions hold true:\\
\begin{enumerate}
  \item Let $1\leqslant p,r\leqslant \infty$, $s\in \mathbb{R}$, then
$$\left\|D^{k}f\right\|_{\dot{B}_{p,r}^{s}}\sim \left\|f\right\|_{\dot{B}_{p,r}^{s+k}}.$$
\item For any $\theta\in (0,1)$, $s_{1},s_{2}\in \mathbb{R}$ such that $s_{1}<s_{2}$, $f\in \dot{B}_{p,r}^{s_{1}}\cap \dot{B}_{p,r}^{s_{2}}$,
then $f\in \dot{B}_{p,r}^{\theta s_{1}+(1-\theta) s_{2}}$ with
\begin{align}\label{interpolation}
\left\|f\right\|_{\dot{B}_{p,r}^{\theta s_{1}+(1-\theta) s_{2}}}\leqslant
C\left\|f\right\|^{\theta}_{\dot{B}_{p,r}^{ s_{1}}}\left\|f\right\|^{(1-\theta)}_{\dot{B}_{p,r}^{s_{2}}}.
\end{align}
\item For any $1\leqslant p_1\leqslant p_2\leqslant \infty$, $1\leqslant r_1\leqslant r_2\leqslant \infty$, $s\in \mathbb{R}$, we have
\begin{align}\label{embedding}
\dot{B}_{p_1,r_1}^s\hookrightarrow\dot{B}_{p_2,r_2}^{s-d(\frac{1}{p_1}-\frac{1}{p_2})}.
\end{align}
\item The operators $\dot{\Delta}_{j}$ and $\dot{S}_{j}$ map $L^p$ into $L^p$ with norms independent of $j$ and $p$.
\item Let $1\leqslant p,r\leqslant \infty$, $s\in (0,\infty)$, then the space $L^{\infty}\cap \dot{B}_{p,r}^{s}$is an algebra with
\begin{align}\label{algebra}
\left\|fg\right\|_{\dot{B}_{p,r}^{s}}\leqslant \frac{C^{s+1}}{s}\left( \left\|f\right\|_{L^{\infty}}\left\|g\right\|_{\dot{B}_{p,r}^{s}}+\left\|g\right\|_{L^{\infty}}\left\|f\right\|_{\dot{B}_{p,r}^{s}}\right).
\end{align}
\item Assume $s_1,s_2\leqslant\frac{d}{p}$ and $s_1+s_2>d {\rm max}(0,\frac{2}{p}-1)$, then we have the following product estimate:
\begin{align}\label{product est}
\|uv\|_{\dot{B}_{p,1}^{s_1+s_2-\frac{d}{p}}}\lesssim \|u\|_{\dot{B}_{p,1}^{s_1}}\|u\|_{\dot{B}_{p,1}^{s_2}}.
\end{align}
\item For any $1\leqslant p,q,r\leqslant\infty$ with $\frac{1}{p}+\frac{1}{q}=\frac{1}{r}$, we have the following commutator estimate:
\begin{align}\label{commu estimate1}
\|[\dot{\Delta}_{j},a]b\|_{L^r}\lesssim 2^{-j}\|\nabla a\|_{L^p}\| b\|_{L^q},
\end{align}
where $[\dot{\Delta}_{j},a]b:=\dot{\Delta}_{j}(ab)-a\dot{\Delta}_{j}b$.
\end{enumerate}
\end{Proposition}

The computation of Theorem \ref{Thm1.1} relies on the regularity estimates of the heat equation stated in the following lemma, which can be founded in \cite{BCD}.
\begin{Lemma}\label{Lemma1}
Let $w_0\in\mathcal{S}'(\mathbb{R}^d)$ and $f\in L^1_{loc}(\mathbb{R}^+;\mathcal{S}'(\mathbb{R}^d))$, the heat equation
\begin{equation*}
\begin{cases}
\partial_{t}w-\Delta w+\nabla p=f,\\
w|_{t=0}=w_{0},
\end{cases}
\end{equation*}
has a unique tempered distribution solution, given by the Duhamel formula:
$$w(t)=e^{t\Delta}w_0+\int_0^te^{(t-\tau)\Delta}f(\tau)d\tau.$$
Moreover, for any $T>0$, $s\in\mathbb{R}$, $1\leqslant r\leqslant q\leqslant m\leqslant\infty$, suppose that $w_0\in\dot{B}_{p,q}^{s+2}$ and $f\in L^r_T\dot{B}_{p,q}^{s+\frac{2}{r}}$, then $w\in L^m_T\dot{B}_{p,q}^{s+2+\frac{2}{m}}$ with the following standard estimate holds:
$$\|w\|_{L^m_T\dot{B}_{p,q}^{s+2+\frac{2}{m}}}\lesssim \|w_0\|_{\dot{B}_{p,q}^{s+2}}+\|f\|_{L^r_T\dot{B}_{p,q}^{s+\frac{2}{r}}}.$$
Furthermore, if $q$ is finite, $w\in \mathcal{C}([0,T];\dot{B}_{p,q}^{s+2})$.
\end{Lemma}
The following lemma will enable us to prove Theorem \ref{Thm1.4} and see \cite{D} for more details.
\begin{Lemma}\label{Lemma2}
 Let $1<p<\infty$, $supp\hat{u}\subset C(0,R_1,R_2)$ (with $0<R_1<R_2$). There exists a constant $c_0$ depending on $R_2, R_1$ such that
 \begin{align}\label{new Bernstein ineq}
 c_0\frac{R_1^2}{p^2}\int_{\mathbb{R}^3}|u|^pdx\leqslant-\frac{1}{p-1}\int_{\mathbb{R}^3}\Delta u|u|^{p-2}udx.
 \end{align}
\end{Lemma}
One may refer to \cite[Lemma 1.6]{D1} for the following basic estimate.
\begin{Lemma}\label{Lemma3}
Let $s>0$, $u\in \dot{B}_{2,1}^s$ and $F\in W_{loc}^{[s]+2,\infty}(\mathbb{R}^3)$ with $F(0)=0$, then $F(u)\in \dot{B}_{2,1}^s$ with
$$\|F(u)\|_{\dot{B}_{2,1}^s}\leqslant C_0(\|u\|_{L^\infty})\|u\|_{\dot{B}_{2,1}^s},$$
where $C_0$ is a function of one variable depending only on $s$ and $F$.
\end{Lemma}

In order to prove the following useful lemma, let us recall the so--called Bony decomposition for the product of $u, v\in \mathcal{S}'_h$:
$$uv=T_uv+T_vu+R(u,v).$$
Above, $T$ denotes the homogeneous paraproduct operator defined as follows:
$$T_uv:=\sum\limits_{j\in \mathbb{Z}}\dot{S}_{j-1}u\dot{\Delta}_{j}v,\quad T_vu:=\sum\limits_{j\in \mathbb{Z}}\dot{S}_{j-1}v\dot{\Delta}_{j}u,$$
and $R$ is the homogeneous remainder operator given by
$$R(u,v)=\sum\limits_{|j-j'|\leqslant 1}\dot{\Delta}_{j}u\dot{\Delta}_{j'}v.$$

The following lemma is critical for justifying Theorem \ref{Thm1.1}.
\begin{Lemma}\label{Lemma4}
 Let $s>0$ and $\rho\in (2,\infty)$, then we have
 \begin{equation}\label{commu estimate}
 \sum\limits_{j\in \mathbb{Z}}2^{js}\|[\dot{\Delta}_{j},b]a\|_{L^2}\lesssim \|a\|_{\dot{B}_{2,1}^{s-\frac{2}{\rho}}}\|b\|_{\dot{B}_{\infty,\infty}^{\frac{2}{\rho}}}+\|b\|_{\dot{B}_{2,1}^{s+1-\frac{2}{\rho}}}\|a\|_{\dot{B}_{\infty,\infty}^{\frac{2}{\rho}}}.
 \end{equation}
\end{Lemma}
\begin{proof}
Employing the Bony decomposition, we have
\begin{align}\label{Bony decomposition}
[\dot{\Delta}_{j},b]a=[\dot{\Delta}_{j},T_b]a+\dot{\Delta}_{j}T_ab-R(b,\dot{\Delta}_{j}a)+\dot{\Delta}_{j}R(b,a)-T_{\dot{\Delta}_{j}a}b.
\end{align}
Using Bernstein's inequality and Young's inequality, one has from \eqref{commu estimate1}
\begin{align*}
\sum\limits_{j\in \mathbb{Z}}2^{js}\|[\dot{\Delta}_{j},T_b]a\|_{L^2}&\leqslant \sum\limits_{j\in \mathbb{Z}}2^{js}\sum \limits_{|j-j'|\leqslant 1}\|[\dot{\Delta}_{j},\dot{S}_{j'-1}b]\dot{\Delta}_{j'}a\|_{L^2}\\
&\lesssim \sum\limits_{j\in \mathbb{Z}}2^{js}\sum \limits_{|j-j'|\leqslant 1}2^{-j}\|\nabla\dot{S}_{j'-1}b\|_{L^\infty}\|\dot{\Delta}_{j'}a\|_{L^2}\\
&\lesssim \sum\limits_{j\in \mathbb{Z}}2^{j(s-1)}\|\dot{\Delta}_{j}a\|_{L^2}\sum \limits_{j'\leqslant j}2^{j'}\|\dot{\Delta}_{j'}b\|_{L^\infty}\\
&\lesssim \sum\limits_{j\in \mathbb{Z}}2^{j(s-\frac{2}{\rho})}\|\dot{\Delta}_{j}a\|_{L^2}\sum \limits_{j'\leqslant j}2^{(j-j')(\frac{2}{\rho}-1)}2^{j'\frac{2}{\rho}}\|\dot{\Delta}_{j'}b\|_{L^\infty}\\
&\lesssim \|a\|_{\dot{B}_{2,1}^{s-\frac{2}{\rho}}}\|b\|_{\dot{B}_{\infty,\infty}^{\frac{2}{\rho}}}.
\end{align*}
A similar argument implies that
\begin{align*}
&\sum\limits_{j\in \mathbb{Z}}2^{js}\|\dot{\Delta}_{j}T_ab-R(b,\dot{\Delta}_{j}a)\|_{L^2}\\
&\leqslant \sum\limits_{j\in \mathbb{Z}}2^{js}\sum \limits_{|j-j'|\leqslant 1}\left(\|\dot{\Delta}_{j'}b\dot{S}_{j'-1}a\|_{L^2}+\sum\limits_{|\nu|\leqslant1}\|\dot{\Delta}_{j'}b\dot{\Delta}_{j'-\nu}\dot{\Delta}_{j}a\|_{L^2}\right)\\
&\lesssim \sum\limits_{j\in \mathbb{Z}}2^{js}\|\dot{\Delta}_{j}b\|_{L^2}\sum \limits_{j'\leqslant j}\|\dot{\Delta}_{j'}a\|_{L^\infty}+\sum\limits_{j\in \mathbb{Z}}2^{js}\|\dot{\Delta}_{j}b\|_{L^2}\|\dot{\Delta}_{j}a\|_{L^\infty}\\
&\lesssim \sum\limits_{j\in \mathbb{Z}}2^{j(s+1-\frac{2}{\rho})}\|\dot{\Delta}_{j}b\|_{L^2}\sum \limits_{j'\leqslant j}2^{(j-j')(\frac{2}{\rho}-1)}2^{j'(\frac{2}{\rho}-1)}\|\dot{\Delta}_{j'}a\|_{L^\infty}\\
&\quad +\sum\limits_{j\in \mathbb{Z}}2^{j(s+1-\frac{2}{\rho})}\|\dot{\Delta}_{j}b\|_{L^2}2^{j(\frac{2}{\rho}-1)}\|\dot{\Delta}_{j}a\|_{L^\infty}\\
&\lesssim\|b\|_{\dot{B}_{2,1}^{s+1-\frac{2}{\rho}}}\|a\|_{\dot{B}_{\infty,\infty}^{\frac{2}{\rho}-1}},
\end{align*}
and
\begin{align*}
&\sum\limits_{j\in \mathbb{Z}}2^{js}\|\dot{\Delta}_{j}R(b,a)-T_{\dot{\Delta}_{j}a}b\|_{L^2}\\
&\leqslant \sum\limits_{j\in \mathbb{Z}}2^{js}\sum \limits_{j\leqslant j'}\left(\sum\limits_{|\nu|\leqslant1}\|\dot{\Delta}_{j}\left(\dot{\Delta}_{j'}b\dot{\Delta}_{j'-\nu}a\right)\|_{L^2}+
\|\dot{\Delta}_{j'}b\dot{S}_{j'-1}\dot{\Delta}_{j}a\|_{L^2}\right)\\
&\lesssim \sum\limits_{j\in \mathbb{Z}}2^{js}\sum \limits_{j\leqslant j'}\|\dot{\Delta}_{j'}a\|_{L^2}\|\dot{\Delta}_{j'}b\|_{L^\infty}+\sum\limits_{j\in \mathbb{Z}}2^{js}\|\dot{\Delta}_{j}a\|_{L^2}\sum \limits_{j\leqslant j'}\|\dot{\Delta}_{j'}b\|_{L^\infty}\\
&\lesssim \sum\limits_{j\in \mathbb{Z}}\sum \limits_{j\leqslant j'}2^{(j-j')s}2^{j'(s-\frac{2}{\rho})}\|\dot{\Delta}_{j'}a\|_{L^2}2^{j'\frac{2}{\rho}}\|\dot{\Delta}_{j'}b\|_{L^\infty}\\
&\quad +\sum\limits_{j\in \mathbb{Z}}2^{j(s-\frac{2}{\rho})}\|\dot{\Delta}_{j}a\|_{L^2}\sum \limits_{j\leqslant j'}2^{(j-j')\frac{2}{\rho}}2^{j'\frac{2}{\rho}}\|\dot{\Delta}_{j'}b\|_{L^\infty}\\
&\lesssim \|a\|_{\dot{B}_{2,1}^{s-\frac{2}{\rho}}}\|b\|_{\dot{B}_{\infty,\infty}^{\frac{2}{\rho}}}.
\end{align*}
Combining these above inequalities with \eqref{Bony decomposition}, we conclude the results of Lemma \ref{Lemma4}.
\end{proof}

\section{Proof of the main result}\label{sec3}
In this section, we prove the main results step by step. More precisely, we prove Theorem \ref{Thm1.1} by using the energy method and the compactness argument. Based on this local results, we prove the global existence provide that the initial is small enough in some suitable space. Later, we establish a blow--up criterion in Theorem \ref{Thm1.3}. Finally, using the blow--up criterion, we obtain the global existence of Theorem \ref{Thm1.4} in Sobolev space.
\subsection{Proof of Theorem \ref{Thm1.1}}\label{sec3.1}
Inspired by \cite[Chapter 10]{BCD} and \cite{DT}, in this subsection, basing on a priori estimate in $E_2(T)$, we prove Theorem \ref{Thm1.1} through a routine compactness argument. We get this a priori estimate in the first step by running energy method on the smooth solutions of the LLB equation after localization in the Fourier space. After that, the uniform regularity of the approximate solutions ensures passing the limit of the nonlinear terms.

\underline{First step: A priori estimates.} The main goal of this part is to prove the following proposition.
\begin{Proposition}\label{proposition3.1}
Suppose $u$ is a smooth solution of the Cauchy problem \eqref{LLB} on $[0,T]\times \mathbb{R}^3$ for some $T>0$. Let $u^L:=e^{t\Delta}u_0$ and $\tilde{u}:=u-u^L$. Define
\begin{align}\label{def1}
\varphi(t)=\|u^L\|_{\dot{B}_{2,1}^{\frac{3}{2}}}+\|u^L\|_{\dot{B}_{2,1}^{\frac{3}{2}}}\|u^L\|_{\dot{B}_{2,1}^{\frac{7}{2}}},
\end{align}
\begin{align}\label{def2}
\psi(t)=\|u^L\|^2_{\dot{B}_{2,1}^{\frac{5}{2}}}+\|u^L\|_{\dot{B}_{2,1}^{\frac{7}{2}}}+\|u^L\|^{\rho}_{\dot{B}_{2,1}^{\frac{2}{\rho}+\frac{3}{2}}}+\|u^L\|^{\frac{\rho}{\rho-1}}
_{\dot{B}_{2,1}^{\frac{7}{2}-\frac{3}{2}}}+1.
\end{align}
Then there exist positive constants $\varepsilon$, $C$ and $C_1$ such that if
\begin{align}\label{condition1}
\int_0^T\varphi(t)e^{C\int_t^T\psi(s)ds}dt< \varepsilon,
\end{align}
we have
\begin{align}
\|\tilde{u}\|_{L^\infty_T\dot{B}_{2,1}^{\frac{3}{2}}}+\int_0^T\frac{C_1}{4}\|\tilde{u}\|_{\dot{B}_{2,1}^{\frac{7}{2}}}+\kappa\|\tilde{u}\|_{\dot{B}_{2,1}^{\frac{3}{2}}}dt\leqslant C\varepsilon,
\end{align}
and
\begin{align}
\|u\|_{L^\infty_T\dot{B}_{2,1}^{\frac{3}{2}}}+\int_0^T\frac{C_1}{4}\|u\|_{\dot{B}_{2,1}^{\frac{7}{2}}}+\kappa\|u\|_{\dot{B}_{2,1}^{\frac{3}{2}}}dt\leqslant C\left(\|u_0\|_{\dot{B}_{2,1}^{\frac{3}{2}}}+\varepsilon\right).
\end{align}
\end{Proposition}

\begin{proof}
In view of Lemma \ref{Lemma1}, we have
\begin{align*}
\|u^L\|_{L^\infty_T\dot{B}_{2,1}^{\frac{3}{2}}}+\|u^L\|_{L^1_T\dot{B}_{2,1}^{\frac{7}{2}}}\lesssim \|u_0\|_{\dot{B}_{2,1}^{\frac{3}{2}}}.
\end{align*}
From the equation \eqref{LLB}, we know that $\tilde{u}$ satisfies the following equation:
\begin{equation}\label{LLB1}
\begin{cases}
\partial_t\tilde{u}=\Delta \tilde{u}-\kappa u+u\times \Delta u-\kappa\mu|u|^2u,\\
\tilde{u}|_{t=0}=0.
\end{cases}
\end{equation}
Applying $\dot{\Delta}_j$ on \eqref{LLB1}, and then taking inner products with $\dot{\Delta}_j\tilde{u}$, one gets from integrating by parts
\begin{align}\label{est2}
&\frac{1}{2}\frac{d}{dt}\|\dot{\Delta}_j\tilde{u}\|_{L^2}^2+C_12^{2j}\|\dot{\Delta}_j\tilde{u}\|_{L^2}^2+\kappa\|\dot{\Delta}_j\tilde{u}\|_{L^2}^2\notag\\
&=\left(-\kappa\dot{\Delta}_ju^L-\kappa\mu\dot{\Delta}_j\left(|u|^2u\right)+\dot{\Delta}_j{\rm div}\left(u\times\nabla u^L\right),\dot{\Delta}_j\tilde{u}\right)-\left(\dot{\Delta}_j\left(u\times\nabla \tilde{u}\right),\dot{\Delta}_j\nabla \tilde{u}\right)\notag\\
&\lesssim\kappa\|\dot{\Delta}_ju^L\|_{L^2}\|\dot{\Delta}_j\tilde{u}\|_{L^2}+\kappa\mu\|\dot{\Delta}_j\left(|u|^2u\right)\|_{L^2}\|\dot{\Delta}_j\tilde{u}\|_{L^2}\notag\\
&\quad+\|\dot{\Delta}_j{\rm div}\left(u\times\nabla u^L\right)\|_{L^2}\|\dot{\Delta}_j\tilde{u}\|_{L^2}+2^{j}\|[\dot{\Delta}_j,u\times]\nabla\tilde{u}\|_{L^2}\|\dot{\Delta}_j\tilde{u}\|_{L^2},
\end{align}
for some $C_1>0$.
Letting $\delta>0$ be a constant and $h_j^2:=\|\dot{\Delta}_j\tilde{u}\|_{L^2}^2+\delta^2$, \eqref{est2} can be recast as
\begin{align*}
&\frac{d}{dt}h_j+C_12^{2j}\frac{h_j^2-\delta^2}{h_j}+\kappa\frac{h_j^2-\delta^2}{h_j}\\
&\lesssim\kappa\|\dot{\Delta}_ju^L\|_{L^2}+\kappa\mu\|\dot{\Delta}_j\left(|u|^2u\right)\|_{L^2}+\|\dot{\Delta}_j{\rm div}\left(u\times\nabla u^L\right)\|_{L^2}+2^{j}\|[\dot{\Delta}_j,u\times]\nabla\tilde{u}\|_{L^2}.
\end{align*}
Setting $\delta\rightarrow 0$, we get
\begin{align*}
&\frac{d}{dt}\|\dot{\Delta}_j\tilde{u}\|_{L^2}+C_12^{2j}\|\dot{\Delta}_j\tilde{u}\|_{L^2}+\kappa\|\dot{\Delta}_j\tilde{u}\|_{L^2}\\
&\lesssim\kappa\|\dot{\Delta}_ju^L\|_{L^2}+\kappa\mu\|\dot{\Delta}_j\left(|u|^2u\right)\|_{L^2}+\|\dot{\Delta}_j{\rm div}\left(u\times\nabla u^L\right)\|_{L^2}+2^{j}\|[\dot{\Delta}_j,u\times]\nabla\tilde{u}\|_{L^2}.
\end{align*}
Then taking $\sum\limits_{j\in \mathbb{Z}}2^{\frac{3}{2}j}$ on both sides, applying \eqref{product est}, Sobolev embedding \eqref{embedding}, Lemma \ref{Lemma3} and Lemma \ref{Lemma4}, we obtain
\begin{align}\label{est3}
&\frac{d}{dt}\|\tilde{u}\|_{\dot{B}_{2,1}^{\frac{3}{2}}}+C_1\|\tilde{u}\|_{\dot{B}_{2,1}^{\frac{7}{2}}}+\kappa\|\tilde{u}\|_{\dot{B}_{2,1}^{\frac{3}{2}}}\notag\\
&\lesssim\kappa\|u^L\|_{\dot{B}_{2,1}^{\frac{3}{2}}}+\kappa\mu\||u|^2u\|_{\dot{B}_{2,1}^{\frac{3}{2}}}+\|{\rm div}\left(u\times\nabla u^L\right)\|_{\dot{B}_{2,1}^{\frac{3}{2}}}+\sum\limits_{j\in \mathbb{Z}}2^{\frac{5}{2}j}\|[\dot{\Delta}_j,u\times]\nabla\tilde{u}\|_{L^2}\notag\\
&\lesssim \|u^L\|_{\dot{B}_{2,1}^{\frac{3}{2}}}+\|u\|_{\dot{B}_{2,1}^{\frac{3}{2}}}+\|u^L\|_{\dot{B}_{2,1}^{\frac{5}{2}}}\|u\|_{\dot{B}_{2,1}^{\frac{5}{2}}}\notag\\
&+\|u^L\|_{\dot{B}_{2,1}^{\frac{7}{2}}}\|u\|_{\dot{B}_{2,1}^{\frac{3}{2}}}+\|\tilde{u}\|_{\dot{B}_{2,1}^{\frac{7}{2}-\frac{2}{\rho}}}\|u\|_{\dot{B}_{2,1}^{\frac{3}{2}+\frac{2}{\rho}}}
+\|u\|_{\dot{B}_{2,1}^{\frac{7}{2}-\frac{2}{\rho}}}\|\tilde{u}\|_{\dot{B}_{2,1}^{\frac{3}{2}+\frac{2}{\rho}}}.
\end{align}
By virtue of the interpolation inequality \eqref{interpolation}, one has
$$\|u\|_{\dot{B}_{2,1}^{\frac{7}{2}-\frac{2}{\rho}}}\lesssim \|u\|^{\frac{1}{\rho}}_{\dot{B}_{2,1}^{\frac{3}{2}}}\|u\|^{1-\frac{1}{\rho}}_{\dot{B}_{2,1}^{\frac{7}{2}}},\;\;\|u\|_{\dot{B}_{2,1}^{\frac{3}{2}+\frac{2}{\rho}}}\lesssim \|u\|^{1-\frac{1}{\rho}}_{\dot{B}_{2,1}^{\frac{3}{2}}}\|u\|^{\frac{1}{\rho}}_{\dot{B}_{2,1}^{\frac{7}{2}}}.$$
Inserting these two inequalities into \eqref{est3}, we infer that
\begin{align}\label{est4}
&\frac{d}{dt}\|\tilde{u}\|_{\dot{B}_{2,1}^{\frac{3}{2}}}+C_1\|\tilde{u}\|_{\dot{B}_{2,1}^{\frac{7}{2}}}+\kappa\|\tilde{u}\|_{\dot{B}_{2,1}^{\frac{3}{2}}}\notag\\
&\leqslant\frac{C_1}{2}\|\tilde{u}\|_{\dot{B}_{2,1}^{\frac{7}{2}}}+C\varphi(t)+C\|\tilde{u}\|_{\dot{B}_{2,1}^{\frac{3}{2}}}\psi(t)
+C\|\tilde{u}\|_{\dot{B}_{2,1}^{\frac{3}{2}}}\|\tilde{u}\|_{\dot{B}_{2,1}^{\frac{7}{2}}},
\end{align}
where $\varphi(t)$, $\psi(t)$ are defined by \eqref{def1} and \eqref{def2}.
Suppose that $\|\tilde{u}\|_{\dot{B}_{2,1}^{\frac{3}{2}}}\leqslant\frac{C_1}{4C}$, by using \eqref{est4}, then we have
\begin{align}\label{est5}
\frac{d}{dt}\|\tilde{u}\|_{\dot{B}_{2,1}^{\frac{3}{2}}}+\frac{C_1}{4}\|\tilde{u}\|_{\dot{B}_{2,1}^{\frac{7}{2}}}+\kappa\|\tilde{u}\|_{\dot{B}_{2,1}^{\frac{3}{2}}}\leqslant C\varphi(t)+C\|\tilde{u}\|_{\dot{B}_{2,1}^{\frac{3}{2}}}\psi(t).
\end{align}
Combining \eqref{est5} and Gronwall's inequality, we get the following inequality:
\begin{align}\label{est6}
&\|\tilde{u}(t)\|_{\dot{B}_{2,1}^{\frac{3}{2}}}+\int_0^T\frac{C_1}{4}\|\tilde{u}(\tau)\|_{\dot{B}_{2,1}^{\frac{7}{2}}}+\kappa\|\tilde{u}(\tau)\|_{\dot{B}_{2,1}^{\frac{3}{2}}}\leqslant C\int_0^T\varphi(t)e^{C\int_t^T\psi(s)ds}dt<C\varepsilon,
\end{align}
while we take $\varepsilon=\frac{C_1}{4C^2}$. A standard bootstrap argument implies the desired conclusion.
\end{proof}

\underline{Second step: Construct approximate solutions.} In order to apply Friedrichs method to construct the approximate solutions, we consider the spectral cut--off operator $\mathbb{E}_n$ defined as
$$\mathcal{F}(\mathbb{E}_nf)(\xi)=1_{\{n^{-1}\leqslant|\xi|\leqslant n\}}(\xi)\hat{f}(\xi).$$
We aim to solve the following modified system:
\begin{equation}\label{approximate system}
\begin{cases}
\partial_tu=\Delta\mathbb{E}_nu-\kappa \mathbb{E}_nu+\mathbb{E}_n(\mathbb{E}_nu\times\Delta\mathbb{E}_nu)-\kappa\mu\mathbb{E}_n(|\mathbb{E}_nu|^2\mathbb{E}_nu),\\
u|_{t=0}=\mathbb{E}_nu_0.
\end{cases}
\end{equation}
One can justify the map
$$u\rightarrow\Delta\mathbb{E}_nu-\kappa \mathbb{E}_nu+\mathbb{E}_n(\mathbb{E}_nu\times\Delta\mathbb{E}_nu)-\kappa\mu\mathbb{E}_n(|\mathbb{E}_nu|^2\mathbb{E}_nu)$$
is locally Lipschitz in terms of the norm $L^2$, by the fact that the operator $\mathbb{E}_n$ maps $L^2$ to all Sobolev spaces. Thus, the ODE \eqref{approximate system} in the Banach space $L^2$ admits a unique maximal solution $u^n\in \mathcal{C}^1([0,T^n);L^2)$. Furthermore, $\mathbb{E}_nu^n$ also admits a solution of \eqref{approximate system} since $\mathbb{E}_n^2=\mathbb{E}_n$. By the uniqueness, one has $\mathbb{E}_nu^n=u^n$, thus $u^n\in \mathcal{C}^1([0,T^n);H^s)$ for any $s\in \mathbb{R}$. For the sake of convenience, we rewrite \eqref{approximate system} as
\begin{equation}\label{approximate system1}
\begin{cases}
\partial_tu^n=\Delta u^n-\kappa u^n+\mathbb{E}_n(u^n\times\Delta u^n)-\kappa\mu\mathbb{E}_n(|u^n|^2u^n),\\
u^n|_{t=0}=\mathbb{E}_nu_0.
\end{cases}
\end{equation}

\underline{Third step: Uniform estimates.} Notice that $\mathbb{E}_n$ doesn't effect on the energy estimates since it's an $L^2$ orthogonal projector. Thus, Proposition \ref{proposition3.1} can be applied on the equation \eqref{approximate system1}. Next, we shall prove that $T^n>T$ for some $T>0$ small enough. For this purpose, we define
$$u^{nL}:=e^{t\Delta}\mathbb{E}_nu_0=\mathbb{E}_ne^{t\Delta}u_0,\;\;\tilde{u}^n:=u^n-u^{nL}.$$
Combining Lemma \ref{Lemma1} and the fact that $\mathbb{E}_n$ is an operator from $\dot{B}_{2,1}^{s}$ to itself with norm $1$, we can infer that condition \eqref{condition1} holds for some $T$ independent of $n$. Assuming $T^n\leqslant T$, and applying Proposition \ref{proposition3.1}, the following estimate holds uniformly for $0\leqslant t< T^n$:
$$\|u^n(t)\|_{\dot{B}_{2,1}^{\frac{3}{2}}}+\int_0^t\frac{C_1}{4}\|u^n(\tau)\|_{\dot{B}_{2,1}^{\frac{7}{2}}}+\kappa\|u^n(\tau)\|_{\dot{B}_{2,1}^{\frac{3}{2}}}dt\leqslant C\left(\|u_0\|_{\dot{B}_{2,1}^{\frac{3}{2}}}+\varepsilon\right).$$
Therefore, the standard continuation criterion for ordinary differential equations implies that $T^n > T$. Furthermore, we have the following estimates:
\begin{align}
\|\tilde{u}\|_{L^\infty_T\dot{B}_{2,1}^{\frac{3}{2}}}+\frac{C_1}{4}\|\tilde{u}\|_{L^1_T\dot{B}_{2,1}^{\frac{7}{2}}}+\kappa\|\tilde{u}\|_{L^1\dot{B}_{2,1}^{\frac{3}{2}}}\leqslant C\varepsilon,
\end{align}
\begin{align}
\|u\|_{L^\infty_T\dot{B}_{2,1}^{\frac{3}{2}}}+\frac{C_1}{4}\|u\|_{L^1\dot{B}_{2,1}^{\frac{7}{2}}}+\kappa\|u\|_{L^1\dot{B}_{2,1}^{\frac{3}{2}}}dt\leqslant C\left(\|u_0\|_{\dot{B}_{2,1}^{\frac{3}{2}}}+\varepsilon\right).
\end{align}

\underline{Fourth step: Existence of a solution.} We will prove that, up to an extraction, the sequence $u^n$ converges to a solution of \eqref{LLB} by the compactness argument. It's obvious that $u^{nL}\rightarrow u^L$ in $E_2(T)$ by Lemma \ref{Lemma1}. Now, we focus on the following system:
\begin{equation}\label{approximate system2}
\begin{cases}
\partial_t\tilde{u}^n=\Delta \tilde{u}^n-\kappa u^n+\mathbb{E}_n(u^n\times\Delta u^n)-\kappa\mu\mathbb{E}_n(|u^n|^2u^n),\\
\tilde{u}^n|_{t=0}=\mathbb{E}_nu_0.
\end{cases}
\end{equation}
Define the cut--off function $\phi_j\in C_c^\infty(\mathbb{R}^3)$ with $supp\phi_j\subset B(0,j+1)$ and $\phi_j\equiv1$ in $B(0,j)$. Then for any $j\in \mathbb{N}$, we have
\begin{align}\label{est7}
\|\partial_t\phi_j\tilde{u}^n\|_{L^2_T\dot{B}_{2,1}^{\frac{1}{2}}}\lesssim&\|\tilde{u}^n\|_{L^2_T\dot{B}_{2,1}^{\frac{5}{2}}}
+\|\phi_ju^n\|_{L^2_T\dot{B}_{2,1}^{\frac{1}{2}}}\notag\\
&+\|u^n\|_{L^\infty_T\dot{B}_{2,1}^{\frac{3}{2}}}\|u^n\|_{L^2_T\dot{B}_{2,1}^{\frac{5}{2}}}
+\|\phi_j{E}_n(|u^n|^2u^n)\|_{L^2_T\dot{B}_{2,1}^{\frac{1}{2}}}.
\end{align}
From \cite[Proposition 2.93]{BCD}, we know that
$$\|\phi_ju^n\|_{L^2_T\dot{B}_{2,1}^{\frac{1}{2}}}\sim \|\phi_ju^n\|_{L^2_TB_{2,1}^{\frac{1}{2}}}\leqslant \|\phi_ju^n\|_{L^2_TB_{2,1}^{\frac{3}{2}}}\thicksim\|\phi_ju^n\|_{L^2_T\dot{B}_{2,1}^{\frac{3}{2}}}\lesssim \|u^n\|_{L^2_T\dot{B}_{2,1}^{\frac{3}{2}}}.$$
Applying a similar argument on the last term in \eqref{est7}, we conclude that $\partial_t\phi_j\tilde{u}^n$ is uniformly bounded in $L^2_T\dot{B}_{2,1}^{\frac{1}{2}}$. Furthermore, a simple computation implies $\partial_t\tilde{u}^n$ is uniformly bounded in $L^1_T\dot{B}_{2,1}^{\frac{3}{2}}$. Since the application $u\rightarrow \phi_ju$ is compact from $\dot{B}_{2,1}^{\frac{3}{2}}$ into $\dot{B}_{2,1}^{\frac{1}{2}}$, and from $\dot{B}_{2,1}^{\frac{5}{2}}$ into $\dot{B}_{2,1}^{\frac{3}{2}}$, we infer that $\phi_j\tilde{u}^n$ is compact in $\mathcal{C}([0,T];\dot{B}_{2,1}^{\frac{1}{2}})$ and in $L^2([0,T];\dot{B}_{2,1}^{\frac{3}{2}})$ by Aubin-Lions theorem. Applying a same argument of \cite[p. 443]{BCD}, we know that there exists $\tilde{u}\in L^\infty_T\dot{B}_{2,1}^{\frac{3}{2}}\cap L^1_T\dot{B}_{2,1}^{\frac{7}{2}}$, such that for any $j\in \mathbb{Z}$,
\begin{align}\label{strong converge}
\phi_j\tilde{u}^n\rightarrow \phi_j\tilde{u}\quad {\rm in}\quad \mathcal{C}([0,T];\dot{B}_{2,1}^{\frac{1}{2}})\cap L^2([0,T];\dot{B}_{2,1}^{\frac{3}{2}}).
\end{align}
Additionally, we have $\tilde{u}|_{t=0}=0$.

Now, we prove $u=u^{L}+\tilde{u}$ solves the equation \eqref{LLB}. For any $\phi\in \mathcal{C}_c^\infty(\mathbb{R}^3)$, $supp\phi\subset B(0,j)$ for same $j\in \mathbb{N}$. As an example, we explain how to pass the limit of the nonlinear term $\mathbb{E}_n(u^n\times\Delta u^n)$. For any $0<t\leqslant T$, we have
\begin{align*}
&\left|\int_0^t\int_{\mathbb{R}^3}(\mathbb{E}_n(u^n\times\Delta u^n)-u\times\Delta u)\phi dxd\tau\right|\\
&\leqslant\left|\int_0^t\int_{\mathbb{R}^3}u^n\times\Delta u^n(\mathbb{E}_n-{\rm Id})\phi dxd\tau\right|+\left|\int_0^t\int_{\mathbb{R}^3}(u^n\times\Delta u^n-u\times\Delta u)\phi \phi_j dxd\tau\right|\\
&\leqslant C\|u^n\times\Delta u^n\|_{L^1_T\dot{B}_{2,1}^{\frac{3}{2}}}\|(\mathbb{E}_n-{\rm Id})\phi\|_{\dot{B}_{2,\infty}^{-\frac{3}{2}}}\\
&\quad+\int_0^t\int_{\mathbb{R}^3}\phi_j(u^n-u)\times\Delta u^n\phi+u\times\Delta \phi_j(u^n-u)\phi dxd\tau\\
&\lesssim \|u^n\|_{L^\infty_T\dot{B}_{2,1}^{\frac{3}{2}}}\|u^n\|_{L^1_T\dot{B}_{2,1}^{\frac{7}{2}}}\|(\mathbb{E}_n-{\rm Id})\phi\|_{\dot{B}_{2,1}^{-\frac{3}{2}}}\\
&\quad+\|\phi_j(u^n-u)\|_{\mathcal{C}_T\dot{B}_{2,1}^{\frac{1}{2}}}\|u^n\|_{L^1_T\dot{B}_{2,1}^{\frac{7}{2}}}\|\phi\|_{\dot{B}_{2,\infty}^{-\frac{1}{2}}}
+\|u\|_{L^2_T\dot{B}_{2,1}^{\frac{3}{2}}}\|\phi_j(u^n-u)\|_{L^2_T\dot{B}_{2,1}^{\frac{3}{2}}}\|\phi\|_{\dot{B}_{2,\infty}^{\frac{1}{2}}}.
\end{align*}
We deduce that the first term tends to zero by the definition of the Besov space. In view of the uniform bounds of $u^n$ and strong convergence of $\phi_j\tilde{u}^n$ in \eqref{strong converge}, the second and third terms tend to zero. The other terms can be treated in a similar way.

\underline{Fifth step: Uniqueness.} Suppose $u_1$, $u_2$ are the solutions of equation \eqref{LLB} on $[0,T]\times\mathbb{R}^3$, and define $\delta u:=u_1-u_2$. Then, $\delta u$ satisfies the following system:
\begin{equation}\label{LLB2}
\begin{cases}
\partial_t\delta u=\Delta \delta u-\kappa \delta u+{\rm div}(\delta u\times\nabla u_1)+{\rm div}(u_2\times\nabla\delta u)-\kappa\mu\delta u(u_1+u_2)u_1-\kappa\mu |u_2|^2\delta u,\\
\delta u|_{t=0}=0.
\end{cases}
\end{equation}
Using a similar argument of a priori estimate in the First Step, we infer that
\begin{align}\label{est8}
&\frac{d}{dt}\|\delta u\|_{\dot{B}_{2,1}^{\frac{3}{2}}}+C_1\|\delta u\|_{\dot{B}_{2,1}^{\frac{7}{2}}}+\kappa\|\delta u\|_{\dot{B}_{2,1}^{\frac{3}{2}}}\notag\\
&\lesssim\|{\rm div}(\delta u\times\nabla u_1)\|_{\dot{B}_{2,1}^{\frac{3}{2}}}+\kappa\mu\|\delta u(u_1+u_2)u_1\|_{\dot{B}_{2,1}^{\frac{3}{2}}}+\kappa\mu\||u_2|^2\delta u\|_{\dot{B}_{2,1}^{\frac{3}{2}}}\notag\\
&\quad+\sum\limits_{j\in \mathbb{Z}}2^{\frac{5}{2}j}\|[\dot{\Delta}_j,u_2\times]\nabla\delta u\|_{L^2}\notag\\
&\lesssim \|\delta u\|_{\dot{B}_{2,1}^{\frac{5}{2}}}\|u_1\|_{\dot{B}_{2,1}^{\frac{5}{2}}}+\|\delta u\|_{\dot{B}_{2,1}^{\frac{3}{2}}}\|u_1\|_{\dot{B}_{2,1}^{\frac{7}{2}}}+\|\delta u\|_{\dot{B}_{2,1}^{\frac{7}{2}-\frac{2}{\rho}}}\|u_2\|_{\dot{B}_{2,1}^{\frac{2}{\rho}+\frac{3}{2}}}\notag\\
&\quad+\|u_2\|_{\dot{B}_{2,1}^{\frac{7}{2}-\frac{2}{\rho}}}\|\delta u\|_{\dot{B}_{2,1}^{\frac{2}{\rho}+\frac{3}{2}}}
+\|\delta u\|_{\dot{B}_{2,1}^{\frac{3}{2}}}(\|u_1\|^2_{\dot{B}_{2,1}^{\frac{3}{2}}}+\|u_2\|^2_{\dot{B}_{2,1}^{\frac{3}{2}}})\notag\\
&\leqslant\frac{C_1}{2}\|\delta u\|_{\dot{B}_{2,1}^{\frac{7}{2}}}+C\|\delta u\|_{\dot{B}_{2,1}^{\frac{3}{2}}}\notag\\
&\quad\left(\|u_1\|^2_{\dot{B}_{2,1}^{\frac{5}{2}}}+\|u_1\|_{\dot{B}_{2,1}^{\frac{7}{2}}}+\| u_2\|^{\frac{\rho}{\rho-1}}_{\dot{B}_{2,1}^{\frac{7}{2}-\frac{2}{\rho}}}+\|u_2\|^\rho_{\dot{B}_{2,1}^{\frac{2}{\rho}+\frac{3}{2}}}+
\|u_1\|^2_{\dot{B}_{2,1}^{\frac{3}{2}}}+\|u_2\|^2_{\dot{B}_{2,1}^{\frac{3}{2}}}\right).
\end{align}
Then, one has $\delta u=0$ on $[0,T]$ by Gronwall's inequality.

\subsection{Proof of Theorem 1.2}\label{sec3.3}

In this subsection, we will adapt the proof of Theorem \ref{Thm1.1} to obtain the global result stated in Theorem \ref{Thm1.2}. First of all, we establish the a priori estimate of the LLB equation \eqref{LLB}.
\begin{Proposition}\label{proposition3.2}
Suppose that $u$ is a smooth solution of the Cauchy problem \eqref{LLB} on $[0,T]\times \mathbb{R}^3$ for some $T>0$. There exist positive constants $\varepsilon$ and $C_1$ such that if $\|u_0\|_{\dot{B}_{2,1}^{\frac{3}{2}}}<\varepsilon$, we have
\begin{align}\label{est20}
\|u\|_{L^\infty_T\dot{B}_{2,1}^{\frac{3}{2}}}+\int_0^T\frac{C_1}{2}\|u\|_{\dot{B}_{2,1}^{\frac{7}{2}}}+\frac{\kappa}{2}\|u\|_{\dot{B}_{2,1}^{\frac{3}{2}}}dt\leqslant\varepsilon.
\end{align}
\end{Proposition}
\begin{proof}
Applying $\dot{\Delta}_j$ on \eqref{LLB}, taking inner products with $\dot{\Delta}_ju$, and using a same argument in Proposition \ref{proposition3.1}, we have
\begin{align}\label{est21}
&\frac{d}{dt}\|\dot{\Delta}_ju\|_{L^2}+C_12^{2j}\|\dot{\Delta}_ju\|_{L^2}+\kappa\|\dot{\Delta}_ju\|_{L^2}\notag\\
&\leqslant\|\dot{\Delta}_j\left(u\times\Delta u\right)\|_{L^2}+\kappa\mu\|\dot{\Delta}_j\left(|u|^2u\right)\|_{L^2}.
\end{align}
Then taking $\sum\limits_{j\in \mathbb{Z}}2^{\frac{3}{2}j}$ on both sides of \eqref{est21}, from \eqref{algebra} and \eqref{product est}, we deduce that
\begin{align}\label{est22}
&\frac{d}{dt}\|u\|_{\dot{B}_{2,1}^{\frac{3}{2}}}+C_1\|u\|_{\dot{B}_{2,1}^{\frac{7}{2}}}+\kappa\|u\|_{\dot{B}_{2,1}^{\frac{3}{2}}}\notag\\
&\lesssim\|u\|_{\dot{B}_{2,1}^{\frac{3}{2}}}\|u\|_{\dot{B}_{2,1}^{\frac{7}{2}}}+\|u\|^3_{\dot{B}_{2,1}^{\frac{3}{2}}}.
\end{align}
Suppose that $\|u(t)\|_{\dot{B}_{2,1}^{\frac{3}{2}}}\leqslant\varepsilon$ for $t\in (0,T]$, then from \eqref{est22} one has
\begin{align}\label{est23}
&\|u(t)\|_{\dot{B}_{2,1}^{\frac{3}{2}}}+\int_0^tC_1\|u(\tau)\|_{\dot{B}_{2,1}^{\frac{7}{2}}}+\kappa\|u(\tau)\|_{\dot{B}_{2,1}^{\frac{3}{2}}}d\tau\notag\\
&\leqslant\|u_0\|_{\dot{B}_{2,1}^{\frac{3}{2}}}+\int_0^tC\varepsilon\|u\|_{\dot{B}_{2,1}^{\frac{7}{2}}}+C\varepsilon^2\|u\|_{\dot{B}_{2,1}^{\frac{3}{2}}}d\tau,
\end{align}
Setting $\varepsilon$ small enough satisfying $C\varepsilon<\frac{C_1}{2}$ and $C\varepsilon^2<\frac{\kappa}{2}$, a standard bootstrap argument implies the result of Proposition \ref{proposition3.2}.
\end{proof}
Taking the process of Theorem \ref{Thm1.1} as an example, we sketch the main points of the rest proof:
\begin{enumerate}
 \item use Friedrichs method and get a sequence $u^n$ of global smooth solutions to the LLB equation \eqref{LLB};
 \item apply Proposition \ref{proposition3.2} to the system of $u_n$ and get the uniform estimate for $u_n$ in $E_2$;
 \item through the compactness argument to obtain, up to extraction, $u_n$ converges to a solution of  the LLB equation supplemented with the initial data $u_0$;
 \item through a similar method of a priori estimate to get the uniqueness of the solution.
\end{enumerate}
Since the justification of those steps are similar to that of Theorem \ref{Thm1.1}, we only detail the different parts.

From the Second Step in Theorem \ref{Thm1.1}, we know that the approximate solution $u^n\in \mathcal{C}^1([0,T^n);H^s)$, for any $s\in \mathbb{R}$. We claim that $T^n=\infty$ under the hypothesis of Theorem \ref{Thm1.2}. By using Proposition \ref{proposition3.2}, one has
\begin{align*}
\|u^n\|_{L^\infty_T\dot{B}_{2,1}^{\frac{3}{2}}}+\int_0^T\frac{C_1}{2}\|u^n(t)\|_{\dot{B}_{2,1}^{\frac{7}{2}}}+\frac{\kappa}{2}\|u^n(t)\|_{\dot{B}_{2,1}^{\frac{3}{2}}}dt\leqslant\varepsilon,
\end{align*}
uniformly for $0<T< T^n$. Thus, the standard continuation criterion for ordinary differential equations implies that $T^n=\infty$.

By a same method of Fourth Step in Theorem \ref{Thm1.1}, we get the global existence of the solution to the LLB equation. The only difference is that we must modify the proof of \cite[p. 443]{BCD} to show $u\in L^1\dot{B}_{2,1}^{\frac{7}{2}}\cap L^1\dot{B}_{2,1}^{\frac{3}{2}}$. Since $u^n$ is uniformly bounded in $L^1\dot{B}_{2,1}^{\frac{7}{2}}\cap L^1\dot{B}_{2,1}^{\frac{3}{2}}$ with
$$\|u^n\|_{L^1\dot{B}_{2,1}^{\frac{7}{2}}\cap L^1\dot{B}_{2,1}^{\frac{3}{2}}}\leqslant C,$$
for some positive constant $C$, we know that $u\in \mathcal{M}(\dot{B}_{2,1}^{\frac{7}{2}})\cap\mathcal{M}(\dot{B}_{2,1}^{\frac{3}{2}})$, which are the bounded measures on $\mathbb{R}^+$ with values in the space $\dot{B}_{2,1}^{\frac{7}{2}}\cap \dot{B}_{2,1}^{\frac{3}{2}}$ and that
$$\int_0^\infty d\|u(t)\|_{\dot{B}_{2,1}^{\frac{7}{2}}}\leqslant C,\quad\int_0^\infty d\|u(t)\|_{\dot{B}_{2,1}^{\frac{3}{2}}}\leqslant C.$$
Thus, we have
\begin{align}\label{est24}
\int_0^Td\|\mathbb{E}_nu(t)\|_{\dot{B}_{2,1}^{\frac{7}{2}}}\leqslant C,\quad\int_0^Td\|\mathbb{E}_nu(t)\|_{\dot{B}_{2,1}^{\frac{3}{2}}}\leqslant C,
\end{align}
uniformly for $T>0$ and $n\in\mathbb{N}$. Since $u\in L^\infty\dot{B}_{2,1}^{\frac{3}{2}}$, one has $\mathbb{E}_nu\in L^1_T\dot{B}_{2,1}^{\frac{7}{2}}\cap L^1_T\dot{B}_{2,1}^{\frac{3}{2}}$. Thus \eqref{est24} can be written as
$$\int_0^T\|\mathbb{E}_nu(t)\|_{\dot{B}_{2,1}^{\frac{7}{2}}}dt\leqslant C,\quad\int_0^T\|\mathbb{E}_nu(t)u(t)\|_{\dot{B}_{2,1}^{\frac{3}{2}}}dt\leqslant C,$$
uniformly for $T>0$ and $n\in\mathbb{N}$. Therefore, we conclude that $u\in L^1\dot{B}_{2,1}^{\frac{7}{2}}\cap L^1\dot{B}_{2,1}^{\frac{3}{2}}$.
The continuity of $u$ can be deduced by using the fact that $\partial_tu\in L^1\dot{B}_{2,1}^{\frac{3}{2}}$.

For the uniqueness, suppose that $u_1$, $u_2$ are solutions to the LLB equation \eqref{LLB}. Without loss of generality, take $u_2$ be the solution built previously and $\|u_2\|_{L^\infty\dot{B}_{2,1}^{\frac{3}{2}}}\leqslant \varepsilon$ provided $\|u_2(0)\|_{\dot{B}_{2,1}^{\frac{3}{2}}}<\varepsilon$. Defining $\delta u:=u_1-u_2$, and applying a same argument of a priori estimate to the system
\begin{equation}\label{LLB3}
\begin{cases}
\partial_t\delta u=\Delta \delta u-\kappa \delta u+\delta u\times\Delta u_1+u_2\times\Delta\delta u-\kappa\mu\delta u(u_1+u_2)u_1-\kappa\mu |u_2|^2\delta u,\\
\delta u|_{t=0}=0,
\end{cases}
\end{equation}
 we get
\begin{align}\label{est26}
&\frac{d}{dt}\|\delta u\|_{\dot{B}_{2,1}^{\frac{3}{2}}}+C_1\|\delta u\|_{\dot{B}_{2,1}^{\frac{7}{2}}}+\kappa\|\delta u\|_{\dot{B}_{2,1}^{\frac{3}{2}}}\notag\\
&\lesssim \|\delta u\|_{\dot{B}_{2,1}^{\frac{3}{2}}}\|u_1\|_{\dot{B}_{2,1}^{\frac{7}{2}}}+\|\delta u\|_{\dot{B}_{2,1}^{\frac{7}{2}}}\|u_2\|_{\dot{B}_{2,1}^{\frac{3}{2}}}
+\|\delta u\|_{\dot{B}_{2,1}^{\frac{3}{2}}}\left(\|u_1\|^2_{\dot{B}_{2,1}^{\frac{3}{2}}}+\|u_2\|^2_{\dot{B}_{2,1}^{\frac{3}{2}}}\right).
\end{align}
By virtue of $\|u_2\|_{L^\infty\dot{B}_{2,1}^{\frac{3}{2}}}\leqslant \varepsilon$, Gronwall's inequality implies that $\delta u=0$.

\subsection{Proof of Theorem \ref{Thm1.3}}\label{sec3.4}
Before proving Theorem \ref{Thm1.3}, we need to establish the local existence of strong solutions to the LLB equation \eqref{LLB} in Sobolev space, which can be stated as the following proposition.
\begin{Proposition}\label{proposition3.3}
For any initial data $u_0\in H^m$ ($m\geqslant2$), there exists some $T>0$, such that the LLB equation \eqref{LLB} has a unique solution
\begin{align*}
u\in\mathcal{C}([0,T];H^m);\quad \nabla u\in L^2([0,T],H^m).
\end{align*}
\end{Proposition}

The routine process of Theorem \ref{Thm1.1} will be applied to the proof of this proposition, thus we only provide the details of a priori estimate and uniqueness. Before proceeding, let's first recall the following commutator estimate.
\begin{Lemma}\label{Lemma5}
For any nonnegative multi-index $\alpha$ whit $|\alpha|\leq m$ and $f,g\in H^{m}$, there holds
$$\left\|\partial_{x}^{\alpha}\left(fg\right)-f\partial_{x}^{\alpha}g\right\|_{L^{2}}\leq C_{s}\left(\left\|D_{x}^{1}f\right\|_{L^{\infty}}\left\|D_{x}^{m-1}g\right\|_{L^{2}}+\left\|g\right\|_{L^{\infty}}\left\|D_{x}^{m}f\right\|_{L^{2}}\right).$$
\end{Lemma}

\begin{Lemma}[A priori estimate]\label{Lemma6}
Suppose $u$ is a smooth solution of the Cauchy problem \eqref{LLB} on $[0,T]\times\mathbb{R}^3$, and $u=u^L+\tilde{u}$ with
\begin{equation}\label{L1}
\begin{cases}
\partial_tu^L=\Delta u^L-\kappa u^L,\\
u^L|_{t=0}=u_0,
\end{cases}
\end{equation}
and
\begin{equation}\label{L}
\begin{cases}
\partial_t\tilde{u}=\Delta \tilde{u}-\kappa \tilde{u}+u\times\Delta u-\kappa\mu |u|^2u,\\
\tilde{u}|_{t=0}=0.
\end{cases}
\end{equation}
There exist positive constants $\sigma$ and $C$ such that if
\begin{align}\label{condition2}
&\sigma^2C\int_0^T\|u^L\|^{\frac{8(m-1)}{4m-7}}_{H^m}+\|\nabla u^L\|^{2}_{H^m}+\sigma^4dt\notag\\
&\quad+C\int_0^T\|u^L\|^{2}_{H^m}\|\nabla u^L\|^{2}_{H^m}+\|u^L\|^{6}_{H^m}dt<\sigma^2,
\end{align}
then we have
\begin{align}\label{est32}
&\|\tilde{u}\|^2_{L_T^\infty H^m}+\|\nabla\tilde{u}\|^2_{L_T^2 H^m}+\kappa\|\tilde{u}\|^2_{L_T^2 H^m}\leqslant\sigma^2,\\
&\|u\|^2_{L_T^\infty H^m}+\|\nabla u\|^2_{L_T^2 H^m}+\kappa\|u\|^2_{L_T^2 H^m}\leqslant \|u_0\|_{H^m}^2+\sigma^2.
\end{align}
\end{Lemma}
\begin{proof}
We have the following estimate for the system \eqref{L1}
\begin{align}\label{est27}
\|u^L\|_{L^\infty_TH^m}^2+2\|\nabla u^L\|_{L^2_TH^m}^2+2\kappa\|u^L\|_{L^2_TH^m}^2=\|u_0\|_{H^m}^2.
\end{align}
For the system \eqref{L}, taking inner product with $\tilde{u}$, one has
\begin{align}\label{est28}
&\frac{1}{2}\frac{d}{dt}\|\tilde{u}\|^2_{L^2}+\|\nabla\tilde{u}\|^2_{L^2}+\kappa\|\tilde{u}\|^2_{L^2}\notag\\
&\lesssim\|\nabla\tilde{u}\|_{L^2}\|u\|_{L^\infty}\|\nabla u^L\|_{L^2}+\|\tilde{u}\|_{L^\infty}\|u\|_{L^\infty}\|u\|^2_{L^2}.
\end{align}
Assuming $1\leqslant|\alpha|\leqslant m$, applying $\partial^\alpha$ on \eqref{L} and then taking inner product with $\partial^\alpha \tilde{u}$, we have
\begin{align}\label{est29}
&\frac{1}{2}\frac{d}{dt}\|\partial^\alpha\tilde{u}\|^2_{L^2}+\|\nabla\partial^\alpha\tilde{u}\|^2_{L^2}+\kappa\|\partial^\alpha\tilde{u}\|^2_{L^2}\notag\\
&=-\left(\partial^\alpha(u\times\nabla \tilde{u}),\partial^\alpha\nabla \tilde{u}\right)-\left(\partial^\alpha(u\times\nabla u^L),\partial^\alpha\nabla \tilde{u}\right)-\kappa\mu\left(\partial^\alpha(|u|^2u),\partial^\alpha\tilde{u}\right)\notag\\
&=-\left([\partial^\alpha,u\times]\nabla \tilde{u},\partial^\alpha\nabla \tilde{u}\right)-\left(\partial^\alpha(u\times\nabla u^L),\partial^\alpha\nabla \tilde{u}\right)-\kappa\mu\left(\partial^\alpha(|u|^2u),\partial^\alpha \tilde{u}\right)\notag\\
&\lesssim\|\nabla\tilde{u}\|_{H^m}\left(\|\nabla u\|_{L^\infty}\|\tilde{u}\|_{H^m}+\|\nabla\tilde{u}\|_{L^\infty}\|u\|_{H^m}+\|u\|_{L^\infty}\|\nabla u^L\|_{H^m}+\|\nabla u^L\|_{L^\infty}\|u\|_{H^m}\right)\notag\\
&\quad+\|\tilde{u}\|_{H^m}\|u\|_{H^m}\|u\|^2_{L^\infty},
\end{align}
where we have used Lemma \ref{Lemma5} and \eqref{algebra} in the last inequality. Taking $\sum\limits_{1\leqslant|\alpha|\leqslant m}$ on \eqref{est29}, combining \eqref{est28} and the Sobolev embedding, we obtain
\begin{align}\label{est30}
&\frac{1}{2}\frac{d}{dt}\|\tilde{u}\|^2_{H^m}+\|\nabla\tilde{u}\|^2_{H^m}+\kappa\|\tilde{u}\|^2_{H^m}\notag\\
&\lesssim \|\nabla\tilde{u}\|_{H^m}\|\nabla\tilde{u}\|_{\dot{H}^{\frac{7}{4}}}\|u^L\|_{H^m}+\|\tilde{u}\|_{H^m}\|u\|^3_{H^m}\notag\\
&\quad+\|\nabla\tilde{u}\|_{H^m}\left(\|\nabla u\|_{H^m}\|\tilde{u}\|_{H^m}+\|\nabla \tilde{u}\|_{H^m}\|\tilde{u}\|_{H^m}+\|u\|_{H^m}\|\nabla u^L\|_{H^m}\right)\notag\\
&\leqslant\frac{1}{4}(\|\nabla\tilde{u}\|^2_{H^m}+\kappa\|\tilde{u}\|^2_{H^m})+C\|\tilde{u}\|^2_{H^m}\left(\|u^L\|^{\frac{8(m-1)}{4m-7}}_{H^m}+\|\nabla\tilde{u}\|^2_{H^m}+\|\nabla u^L\|^2_{H^m}\right)\notag\\
&\quad+C\left(\|\tilde{u}\|^6_{H^m}+\|u^L\|^6_{H^m}+\|u^L\|^2_{H^m}\|\nabla u^L\|^2_{H^m}\right),
\end{align}
where we have used the interpolation inequality:
$$\|\nabla\tilde{u}\|_{\dot{H}^{\frac{7}{4}}}\lesssim \|\nabla\tilde{u}\|^{\frac{3}{4(m-1)}}_{H^m}\|\nabla\tilde{u}\|^{\frac{4m-7}{4(m-1)}}_{H^1}
\lesssim\|\nabla\tilde{u}\|^{\frac{3}{4(m-1)}}_{H^m}\|\tilde{u}\|^{\frac{4m-7}{4(m-1)}}_{H^m}.$$
There exists a positive constant $\sigma$ with $C\sigma^2<\frac{1}{4}$. Suppose that $\|\tilde{u}(t)\|^2_{H^m}\leqslant\sigma^2$ for all $0<t<T$, then we have
\begin{align}\label{est31}
&\|\tilde{u}(t)\|^2_{H^m}+\int_0^T\|\nabla\tilde{u}\|^2_{H^m}+\kappa\|\tilde{u}\|^2_{H^m}dt\notag\\
&\leqslant\sigma^2C\int_0^T\|u^L\|^{\frac{8(m-1)}{4m-7}}_{H^m}+\|\nabla u^L\|^{2}_{H^m}+\sigma^4dt\notag\\
&\quad+C\int_0^T\|u^L\|^{2}_{H^m}\|\nabla u^L\|^{2}_{H^m}+\|u^L\|^{6}_{H^m}dt\notag\\
&<\sigma^2.
\end{align}
Bootstrap principle implies the conclusion of Lemma \ref{Lemma6}.
\end{proof}
\begin{Lemma}[Uniqueness]\label{Lemma7}
Suppose $u_i$ ($i=1,2$) are solutions of the Cauchy problem \eqref{LLB} on $[0,T]\times\mathbb{R}^3$ with the same initial data, and such that
\begin{align*}
u_i\in\mathcal{C}([0,T];H^m);\quad \nabla u_i\in L^2([0,T],H^m),
\end{align*}
then $u_1=u_2$.
\end{Lemma}
\begin{proof}
Suppose $\delta u:=u_1-u_2$, then $\delta u$ satisfies the equation \eqref{LLB3}. Letting \eqref{LLB3} take inner product with $\delta u$, integrating by parts and using the Sobolev embedding, we have
\begin{align}\label{est33}
&\frac{1}{2}\frac{d}{dt}\|\delta u\|^2_{L^2}+\|\nabla\delta u\|^2_{L^2}+\kappa\|\delta u\|^2_{L^2}\notag\\
&=(u_2\times\Delta \delta u,\delta u)-\kappa\mu\int_{\mathbb{R}^3}|\delta u|^2(u_1+u_2)u_1+|u_2|^2|\delta u|^2dx\notag\\
&\lesssim\|\nabla\delta u\|_{L^2}\|\delta u\|_{L^2}\|\nabla u_2\|_{L^\infty}+\|\delta u\|_{L^2}^2\|u_1\|_{L^\infty}\|u_1+u_2\|_{L^\infty}\notag\\
&\leqslant\frac{1}{2}\|\nabla\delta u\|_{L^2}^2+C\|\delta u\|_{L^2}^2\left(\|\nabla u_2\|_{H^m}^2+\|u_1\|_{H^m}^2+\|u_2\|_{H^m}^2\right),
\end{align}
thus from Gronwall's inequality we conclude that $\delta u=0$ on $[0,T]$.
\end{proof}

Now, we are ready to prove Theorem \ref{Thm1.3}. Suppose $u$ is the unique solution constructed in Proposition \ref{proposition3.3}.
Applying $\dot{\Delta}_j$ on \eqref{LLB1}, then taking inner products with $\dot{\Delta}_ju$, we have
\begin{align*}
&\frac{1}{2}\frac{d}{dt}\|\dot{\Delta}_ju\|^2_{L^2}+C_{1}2^{2j}\|\dot{\Delta}_ju\|^2_{L^2}+\kappa\|\dot{\Delta}_ju\|^2_{L^2}\\
&\leqslant \int_{\mathbb{R}^3}\dot{\Delta}_j(u\times\Delta u)\dot{\Delta}_ju-\kappa\mu\dot{\Delta}_j(|u|^2u)\dot{\Delta}_judx.
\end{align*}
Taking $\sum\limits_{j\in \mathbb{Z}}2^{2jm}$ on both sides, we get
\begin{align}\label{est9}
&\frac{1}{2}\frac{d}{dt}\|u\|^2_{\dot{H}^m}+C_{1}\|u\|^2_{\dot{H}^{m+1}}+\kappa\|u\|^2_{\dot{H}^m}\notag\\
&\leqslant \sum\limits_{j\in \mathbb{Z}}2^{2jm}\int_{\mathbb{R}^3}\dot{\Delta}_j(u\times\Delta u)\dot{\Delta}_ju-\kappa\mu\dot{\Delta}_j(|u|^2u)\dot{\Delta}_judx\notag\\
&:=I_1+I_2.
\end{align}
For the second term, from \eqref{algebra} and the Sobolev embedding, one has
\begin{align}\label{est10}
&I_2\leqslant \|u\|_{\dot{H}^m}\||u|^2u\|_{\dot{H}^m}\notag\\
&\lesssim \|u\|_{\dot{H}^m}\left(\||u|^2\|_{L^\infty}\|u\|_{\dot{H}^m}+\||u|^2\|_{\dot{H}^m}\|u\|_{L^\infty}\right)\notag\\
&\lesssim \|u\|^2_{\dot{H}^m}\||u|^2\|_{\dot{B}^{\frac{3}{p}}_{p,1}}+\|u\|^2_{\dot{H}^m}\|u\|^2_{\dot{B}^{\frac{3}{p}}_{p,1}}\notag\\
&\lesssim \|u\|^2_{\dot{H}^m}\|u\|^2_{\dot{B}^{\frac{3}{p}}_{p,1}}.
\end{align}
In view of Bony decomposition in Section \ref{sec2}, $I_1$ can be recast as
\begin{align}\label{est11}
I_1&=-\sum\limits_{j\in \mathbb{Z}}2^{2jm}\int_{\mathbb{R}^3}\dot{\Delta}_j(u\times\nabla u)\dot{\Delta}_j\nabla udx\notag\\
&=-\sum\limits_{j\in \mathbb{Z}}\sum\limits_{|j-j'|\leqslant1}2^{2jm}\int_{\mathbb{R}^3}\dot{\Delta}_j\left(\dot{S}_{j'-1}u\times\dot{\Delta}_{j'}\nabla u\right)\dot{\Delta}_j\nabla udx\notag\\
&\quad-\sum\limits_{j\in \mathbb{Z}}\sum\limits_{|j-j'|\leqslant1}2^{2jm}\int_{\mathbb{R}^3}\dot{\Delta}_j\left(\dot{\Delta}_{j'}u\times\dot{S}_{j'-1}\nabla u\right)\dot{\Delta}_j\nabla udx\notag\\
&\quad-\sum\limits_{j\in \mathbb{Z}}\sum\limits_{j\leqslant j'}2^{2jm}\int_{\mathbb{R}^3}\dot{\Delta}_j\left(\sum\limits_{|\nu|\leqslant1}\dot{\Delta}_{j'}u\times\dot{\Delta}_{j'-\nu}\nabla u\right)\dot{\Delta}_j\nabla udx\notag\\
&:=I_{11}+I_{12}+I_{13}.
\end{align}
Firstly, we deal with the term $I_{11}$ as
\begin{align}\label{est12}
I_{11}&=-\sum\limits_{j\in \mathbb{Z}}\sum\limits_{|j-j'|\leqslant1}2^{2jm}\int_{\mathbb{R}^3}[\dot{\Delta}_j,\dot{S}_{j'-1}u\times]\dot{\Delta}_{j'}\nabla u\dot{\Delta}_j\nabla udx\notag\\
&\quad-\sum\limits_{j\in \mathbb{Z}}\sum\limits_{|j-j'|\leqslant1}2^{2jm}\int_{\mathbb{R}^3}\dot{S}_{j-1}u\times\dot{\Delta}_{j}\dot{\Delta}_{j'}\nabla u\dot{\Delta}_j\nabla udx\notag\\
&\quad+\sum\limits_{j\in \mathbb{Z}}\sum\limits_{|j-j'|\leqslant1}2^{2jm}\int_{\mathbb{R}^3}\left(\dot{S}_{j-1}-\dot{S}_{j'-1}\right)u\times\dot{\Delta}_{j}\dot{\Delta}_{j'}\nabla u\dot{\Delta}_j\nabla udx\notag\\
&:=I_{111}+I_{112}+I_{113}.
\end{align}
Noted that $I_{112}=0$. In view of the commutator estimate \eqref{commu estimate1}, Bernstein's inequality and Young's inequality, we have
\begin{align}\label{est13}
I_{111}&\lesssim\sum\limits_{j\in \mathbb{Z}}\sum\limits_{|j-j'|\leqslant1}2^{2jm}2^{-j}\|\nabla\dot{S}_{j'-1}u\|_{L^\infty}\|\dot{\Delta}_{j'}\nabla u\|_{L^2}\|\dot{\Delta}_{j}\nabla u\|_{L^2}\notag\\
&\lesssim\sum\limits_{j\in \mathbb{Z}}2^{j(2m+1)}\|\dot{\Delta}_{j} u\|^2_{L^2}\sum\limits_{j'\leqslant j}2^{j'}\|\dot{\Delta}_{j'} u\|_{L^\infty}\notag\\
&\lesssim\sum\limits_{j\in \mathbb{Z}}\left(2^{j(m+1)}\|\dot{\Delta}_{j} u\|_{L^2}\right)^\delta\left(2^{jm}\|\dot{\Delta}_{j} u\|_{L^2}\right)^{2-\delta}
\sum\limits_{j'\leqslant j}2^{(j-j')(1-\delta)}2^{j'(2-\delta)}\|\dot{\Delta}_{j'} u\|_{L^\infty}\notag\\
&\lesssim\|u\|^{\delta}_{\dot{H}^{m+1}}\|u\|_{\dot{H}^{m}}^{2-\delta}\|u\|_{\dot{B}^{2-\delta}_{\infty,\infty}}\notag\\
&\leqslant\frac{C_1}{8}\|u\|^{2}_{\dot{H}^{m+1}}+C\|u\|_{\dot{H}^{m}}^{2}\|u\|^{\frac{2}{2-\delta}}_{\dot{B}^{2-\delta}_{\infty,\infty}},
\end{align}
for any $\delta\in(1,2)$.
For the term $I_{113}$, by Bernstein's inequality and H\"{o}lder's inequality, we have
\begin{align}\label{est14}
I_{113}&\lesssim\sum\limits_{j\in \mathbb{Z}}2^{j(2m+2)}\|\dot{\Delta}_{j} u\|^2_{L^2}\|\dot{\Delta}_{j} u\|_{L^\infty}\notag\\
&\lesssim\sum\limits_{j\in \mathbb{Z}}\left(2^{j(m+1)}\|\dot{\Delta}_{j} u\|_{L^2}\right)^\delta\left(2^{jm}\|\dot{\Delta}_{j} u\|_{L^2}\right)^{2-\delta}
2^{j(2-\delta)}\|\dot{\Delta}_{j} u\|_{L^\infty}\notag\\
&\lesssim\|u\|^{\delta}_{\dot{H}^{m+1}}\|u\|_{\dot{H}^{m}}^{2-\delta}\|u\|_{\dot{B}^{2-\delta}_{\infty,\infty}}\notag\\
&\leqslant\frac{C_1}{8}\|u\|^{2}_{\dot{H}^{m+1}}+C\|u\|_{\dot{H}^{m}}^{2}\|u\|^{\frac{2}{2-\delta}}_{\dot{B}^{2-\delta}_{\infty,\infty}}.
\end{align}
The estimate of the term $I_{12}$ is similar to that of $I_{111}$, so from \eqref{est12}--\eqref{est14}, we have
\begin{align}\label{est16}
I_{11}+I_{12}&\leqslant\frac{3C_1}{8}\|u\|^{2}_{\dot{H}^{m+1}}+C\|u\|_{\dot{H}^{m}}^{2}\|u\|^{\frac{2}{2-\delta}}_{\dot{B}^{2-\delta}_{\infty,\infty}}.
\end{align}
Finally, we deal with the term $I_{13}$. Suppose that $\delta_1\in(0,1)$, $\delta_2\in(0,1)$ and $\delta=\delta_1+\delta_2$, a similar argument in \eqref{est13} implies the following estimate:
\begin{align}\label{est15}
I_{13}&\lesssim\sum\limits_{j\in \mathbb{Z}}2^{j(2m+1)}\|\dot{\Delta}_{j} u\|_{L^2}\sum\limits_{j\leqslant j'}2^{j'}\|\dot{\Delta}_{j'} u\|_{L^2}\|\dot{\Delta}_{j'} u\|_{L^\infty}\notag\\
&\lesssim\sum\limits_{j\in \mathbb{Z}}\left(2^{j(m+1)}\|\dot{\Delta}_{j} u\|_{L^2}\right)^{\delta_1}\left(2^{jm}\|\dot{\Delta}_{j} u\|_{L^2}\right)^{1-\delta_1}
\sum\limits_{j\leqslant j'}2^{(j-j')(m+1-\delta_1)}\notag\\
&\quad\left(2^{j'(m+1)}\|\dot{\Delta}_{j'} u\|_{L^2}\right)^{\delta_2}\left(2^{j'm}\|\dot{\Delta}_{j'} u\|_{L^2}\right)^{1-\delta_2}2^{j'(2-\delta_1-\delta_2)}\|\dot{\Delta}_{j'} u\|_{L^\infty}\notag\\
&\lesssim\|u\|^{\delta_1+\delta_2}_{\dot{H}^{m+1}}\|u\|_{\dot{H}^{m}}^{2-\delta_1-\delta_2}\|u\|_{\dot{B}^{2-\delta_1-\delta_2}_{\infty,\infty}}\notag\\
&\leqslant\frac{C_1}{8}\|u\|^{2}_{\dot{H}^{m+1}}+C\|u\|_{\dot{H}^{m}}^{2}\|u\|^{\frac{2}{2-\delta}}_{\dot{B}^{2-\delta}_{\infty,\infty}}.
\end{align}
From \eqref{est10}, \eqref{est11}, \eqref{est16} and \eqref{est15}, one has
\begin{align}\label{est17}
I_1+I_2\leqslant\frac{C_1}{2}\|u\|^{2}_{\dot{H}^{m+1}}+C\|u\|_{\dot{H}^{m}}^{2}\left(\|u\|^2_{\dot{B}^{\frac{3}{p}}_{p,1}}
+\|u\|^{\frac{2}{2-\delta}}_{\dot{B}^{2-\delta}_{\infty,\infty}}\right).
\end{align}
Thus, combining \eqref{est9}, \eqref{est17} with the following conservation law
\begin{align}\label{conservation law}
\frac{1}{2}\frac{d}{dt}\|u\|^2_{L^2}+\|\nabla u\|^2_{L^2}+\kappa\|u\|^2_{L^2}+\kappa\mu\|u\|^4_{L^4}=0,
\end{align}
we conclude that
\begin{align}\label{est18}
&\frac{1}{2}\frac{d}{dt}\|u\|^2_{H^m}+C_{2}\|\nabla u\|^2_{H^{m}}+\kappa\|u\|^2_{H^m}+\kappa\mu\|u\|^4_{L^4}\notag\\
&\leqslant C\|u\|_{\dot{H}^{m}}^{2}\left(\|u\|^2_{\dot{B}^{\frac{3}{p}}_{p,1}}
+\|u\|^{\frac{2}{2-\delta}}_{\dot{B}^{2-\delta}_{\infty,\infty}}\right),
\end{align}
for some $C_2>0$. By using Gronwall's inequality, we have
\begin{align}\label{est19}
&\|u(t)\|^2_{H^m}+\int_0^tC_{2}\|\nabla u(\tau)\|^2_{H^{m}}+\kappa\|u(\tau)\|^2_{H^m}+\kappa\mu\|u(\tau)\|^4_{L^4}d\tau\notag\\
&\leqslant \|u_0\|_{\dot{H}^{m}}^{2}e^{C\int_0^t\|u(\tau)\|^2_{\dot{B}^{\frac{3}{p}}_{p,1}}
+\|u(\tau)\|^{\frac{2}{2-\delta}}_{\dot{B}^{2-\delta}_{\infty,\infty}}d\tau},
\end{align}
which implies the desired results of  Theorem \ref{Thm1.3}.

\subsection{Proof of Theorem \ref{Thm1.4}}\label{sec3.5}
In this subsection, we will prove Theorem \ref{Thm1.4} through the blow--up criterion established in Theorem \ref{Thm1.3}. Taking $\dot{\Delta}_{j}$ on \eqref{LLB}, we get
\begin{align}\label{Bony equation}
\partial_t\dot{\Delta}_{j}u=\Delta \dot{\Delta}_{j}u-\kappa \dot{\Delta}_{j}u+\dot{\Delta}_{j}(u\times \Delta u)-\kappa\mu\dot{\Delta}_{j}(|u|^2u).
\end{align}
Letting \eqref{Bony equation} take inner product with $p|\dot{\Delta}_{j}u|^{p-2} \dot{\Delta}_{j}u$, one has
\begin{align}\label{est34}
&\frac{d}{dt}\|\dot{\Delta}_{j}u\|_{L^p}^p-p\int_{\mathbb{R}^3}|\dot{\Delta}_{j}u|^{p-2} \dot{\Delta}_{j}u\Delta \dot{\Delta}_{j}udx+\kappa p\|\dot{\Delta}_{j}u\|_{L^p}^p\notag\\
&=p\int_{\mathbb{R}^3}|\dot{\Delta}_{j}u|^{p-2} \dot{\Delta}_{j}u(\dot{\Delta}_{j}(u\times \Delta u)-\kappa\mu\dot{\Delta}_{j}(|u|^2u))dx.
\end{align}
Plugging \eqref{new Bernstein ineq} into \eqref{est34}, we have
\begin{align}
&\frac{d}{dt}\|\dot{\Delta}_{j}u\|_{L^p}^p+\frac{p-1}{p}c_02^{2j}\|\dot{\Delta}_{j}u\|_{L^p}^p+\kappa p\|\dot{\Delta}_{j}u\|_{L^p}^p\notag\\
&\leqslant p\|\dot{\Delta}_{j}u\|_{L^p}^{p-1}(\|\dot{\Delta}_{j}(u\times \Delta u)\|_{L^p}+\kappa\mu\|\dot{\Delta}_{j}(|u|^2u)\|_{L^p}).
\end{align}
Letting $h_j^p=\|\dot{\Delta}_{j}u\|_{L^p}^p+\delta^p$, and then taking $\delta$ tends to $0$, we infer that
\begin{align}\label{est25}
&\frac{d}{dt}\|\dot{\Delta}_{j}u\|_{L^p}+\frac{p-1}{p^2}c_02^{2j}\|\dot{\Delta}_{j}u\|_{L^p}+\kappa \|\dot{\Delta}_{j}u\|_{L^p}\notag\\
&\leqslant\|\dot{\Delta}_{j}(u\times \Delta u)\|_{L^p}+\kappa\mu\|\dot{\Delta}_{j}(|u|^2u)\|_{L^p}.
\end{align}
Then taking $\sum\limits_{j\in \mathbb{Z}}2^{j\frac{3}{p}}$ on both sides of \eqref{est25},  we obtain that from \eqref{product est}
\begin{align}\label{est1}
\frac{d}{dt}\|u\|_{\dot{B}_{p,1}^{\frac{3}{p}}}+C_3\|u\|_{\dot{B}_{p,1}^{\frac{3}{p}+2}}+\kappa \|u\|_{\dot{B}_{p,1}^{\frac{3}{p}}}
&\leqslant  \|u\times \Delta u\|_{\dot{B}_{p,1}^{\frac{3}{p}}}+\kappa\mu\|\dot{\Delta}_{j}(|u|^2u)\|_{\dot{B}_{p,1}^{\frac{3}{p}}}\notag\\
&\leqslant C(\|u\|_{\dot{B}_{p,1}^{\frac{3}{p}}}\|u\|_{\dot{B}_{p,1}^{\frac{3}{p}+2}}+\|u\|_{\dot{B}_{p,1}^{\frac{3}{p}}}^3),
\end{align}
where $C_3=\frac{p-1}{p^2}c_0$. Choosing $\eta$ small enough such that $\eta\leqslant\frac{C_3}{2C}$ and $C\eta^2<\frac{\kappa}{2}$, if $\|u\|_{L^\infty_{T}\dot{B}_{p,1}^{\frac{3}{p}}}\leqslant\eta$, and applying \eqref{est1}, one can deduce that
$$\|u\|_{L^\infty_{T}\dot{B}_{p,1}^{\frac{3}{p}}}+\frac{C_3}{2}\|u\|_{L^1_{T}\dot{B}_{p,1}^{\frac{3}{p}+2}}+\frac{\kappa}{2}\|u\|_{L^1_{T}\dot{B}_{p,1}^{\frac{3}{p}}}
\leqslant\|u_0\|_{\dot{B}_{p,1}^{\frac{3}{p}}}<\eta.
$$
By virtue of the bootstrap principle, we conclude that
$$\|u\|_{L^\infty_{T}\dot{B}_{p,1}^{\frac{3}{p}}}+\frac{C_3}{2}\|u\|_{L^1_{T}\dot{B}_{p,1}^{\frac{3}{p}+2}}+\frac{\kappa}{2}\|u\|_{L^1_{T}\dot{B}_{p,1}^{\frac{3}{p}}}
\leqslant\eta.
$$
By using Sobolev inequality, interpolation inequality and Theorem \ref{Thm1.3}, we can get the desired results of Theorem \ref{Thm1.4}.

\section*{Acknowledgments}  H. Wang was supported by the National Natural Science Foundation of China (No. 11901066), the Natural Science Foundation of Chongqing (No. cstc2019jcyj-msxmX0167), and projects Nos. 2019CDXYST0015, 2020CDJQY-A040 supported by the Fundamental Research Funds for the Central Universities.

\end{document}